\documentclass[12pt]{amsart}
\usepackage[utf8]{inputenc}
\usepackage[top=1in, left=1in, right=1in, bottom=1in]{geometry}
\usepackage{libertine}
\usepackage[libertine]{newtxmath}
\usepackage{xfrac, relsize}
\usepackage{faktor}
\usepackage{color,amsmath,mathtools,graphicx}
\usepackage{bbold}
\usepackage{epsfig,fancyhdr,xcolor}
\usepackage{float}
\usepackage{bm}
\usepackage{setspace}
\usepackage{epstopdf}
\usepackage{import}
\usepackage{lineno}
\usepackage{stackrel,dsfont}
\usepackage[allcolors = blue,colorlinks]{hyperref}
\usepackage[abs]{overpic}
\usepackage[center]{caption}
\usepackage{subcaption}
\usepackage{tikz-cd}
\usepackage{enumitem}
\usepackage{wasysym}
\usepackage{marvosym}

\newtheorem{theorem}{Theorem}[section]

\newtheorem{definition}[theorem]{Definition}

\newtheorem{proposition}[theorem]{Proposition}

\newtheorem{corollary}[theorem]{Corollary}

\title{A note on a short proof of the parallelizability of orientable $3$-manifolds}

\author{Dionne Ibarra}
\address{School of Mathematics, 9 Rainforest Walk, Floor 4, Monash University, VIC 3800, Australia}
\email{dionne.ibarra@monash.edu}


\begin{document}

\begin{abstract}
We survey, complete, and modify a proof, involving knot theory, of Stiefel's theorem that all orientable $3$-manifolds are parallelizable. The completion of the proof is done by using the relationship between the tangent bundle and normal bundle of manifolds with non-trivial boundary and on stably parallelizable and parallelizable manifolds. We end with a remark on $7$-manifolds and present J. Korba\v{s}' example of a non-parallelizable $7$-manifold.
\end{abstract} 

\maketitle

\tableofcontents

 \section*{Acknowledgements}

This work was supported by the Australian Research Council grant DP210103136.
\section{Introduction}
Different approaches to the proof of Stiefel's theorem, all $3$-manifolds are parallelizable, have been done, see \cite{BZ, BL, DGGK, Gon, Gei, Gon, Kir2, FoMa, Whi}.
This survey focuses on the connections between knot theory and the parallelizability of $3$-manifolds discussed in \cite{FoMa}. A. T. Fomenko and S. V. Matveev in \cite{FoMa} used knot theory, in particular integral surgery, to prove that all $3$-manifolds are stably parallelizable.  R. Benedetti and P. Lisca in \cite{BL} used A. T. Fomenko and S. V. Matveev's proof and quasi-framings to give a proof of Stiefel's theorem. Furthermore, S. Durst, H. Geiges, J. Gonzalo Pérez, and M. Kegel remarked in \cite{DGGK} that quasi-framings can be replaced by defining a quaternionic structure on the tangent
bundle of the parallelizable $4$-manifold bounding the $3$-manifold.  In this note we will give a different approach to \cite{BL} and \cite{DGGK} by replacing quasi-framings with the span of a manifold and slightly altering the proof in \cite{FoMa} by giving a different doubling argument when considering a compact $3$-manifold with non-trivial boundary. 
Throughout this paper all manifolds are assumed to be connected compact $3$-manifolds unless otherwise stated. 

In Section \ref{section:integralsurgery}, we will give a brief description of Dehn surgery and integral surgery then list major results from W. B. R. Lickorish \cite{Lic1}, A. H. Wallace \cite{Wal}, R. Thom \cite{Thom}, R, Kirby \cite{Kir}, and J. Milnor \cite{Mil} that connect integral surgery to $3$-manifolds and parallelizable $4$-manifolds. We end the section by discussing how A. T. Fomenko and S. V. Matveev in \cite{FoMa} use these results as a foundation to prove that all closed orientable $3$-manifolds are stably parallelizable. In Section \ref{section:parallelizable}, we turn our attention to tangent bundles and the span of a manifold. We then discuss a result by D. Husemoller in \cite{Hus} on the relationship between the tangent bundle of a manifold and it's boundary. We end the section by extracting a corollary from the proposition. In Section \ref{shortproof}, we present a short proof of Stiefel's theorem that all orientable $3$-manifolds are parallelizable. In Section \ref{7manifolds} we make a remark on $7$-manifolds and present J. Korba\v{s}' example of a non-parallelizable $7$-manifold.

\section{Integral surgery}\label{section:integralsurgery}

Dehn surgery is a well-studied technique that uses knot theory to understand $3$-manifolds. A special case of this technique was introduced in 1910 by M. Dehn \cite{Deh} it was then generalized in the work of R. H. Bing in the late 1950's \cite{Bin1, Bin2}. In this section we will give an account of the various important theorems that provide a clear picture of the connections between knot theory, 3-manifolds, and parallelizable $4$-manifolds.

\begin{definition}
\label{Dehnsurgery}
\textbf{Dehn surgery} along a knot $K$ in an orientable 3-manifold, $M$, is the process of taking the tubular neighborhood of the underlying unframed knot $ K$, $U( K)$, which is homeomorphic to $D^2 \times S^1$, and removing it from $M$ so that we obtain two manifolds: $X_K = M \backslash int( U( K))$ and $V = U( K) \cong D^2 \times S^1$. We use an orientation reversing homeomorphism $\varphi: \partial V \to \partial X_K$ determined by a curve, $c$, in $\partial V$ to add $V$ back into $X_K$, that is, adding a $2$-handle  along $c$ and then cap it off with a 3-ball, see Figure \ref{attachingmap}. The result $X_K \cup_{\varphi} V$ is a closed orientable manifold, we say this manifold is obtained by Dehn surgery along the knot from $M$. We may define Dehn surgery along a link by applying the process to each component where for each component there is an accompanying curve on the tubular neighborhood of the component.
\end{definition}

\begin{figure}[ht]
$$\vcenter{\hbox{
\begin{overpic}[scale = .4]{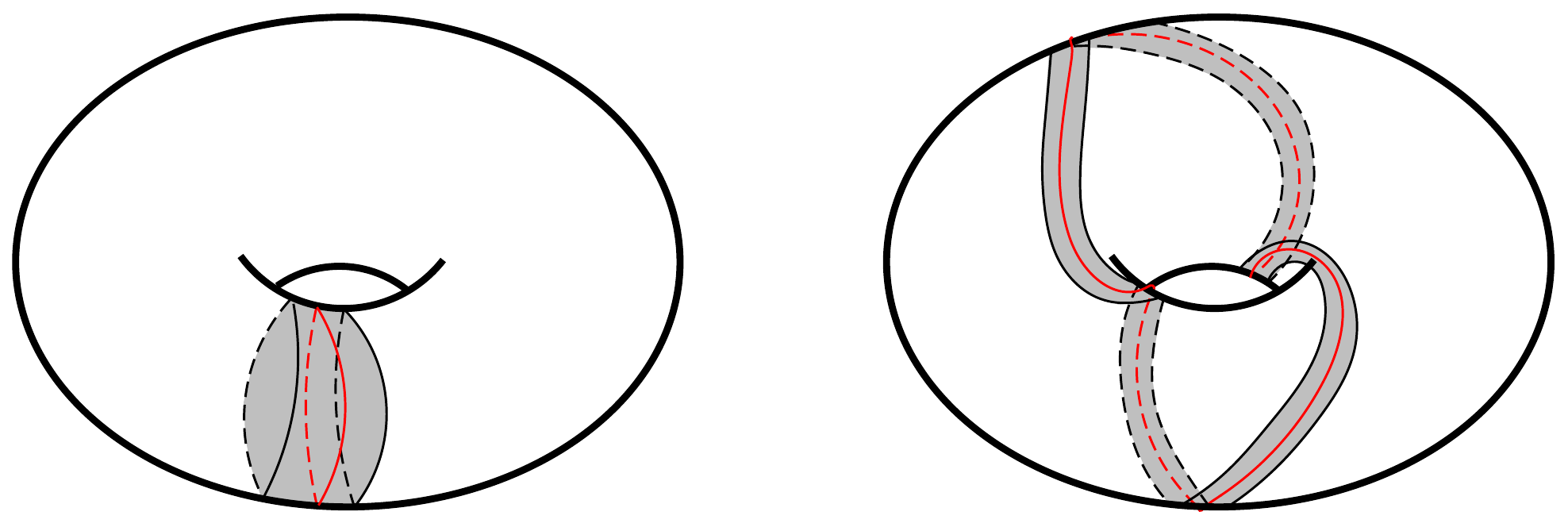}
\put(41, -6){$\mu$}
\put(200, 20){$c$}
\put(107, 35){$\xrightarrow{\varphi}$}
\end{overpic} }}$$
\caption{An illustration of attaching a $2$-handle from $V = (2-handle) \cup D^3$ along $c = \varphi(\mu)$ before capping it off with $D^3$.} \label{attachingmap}
\end{figure}

\begin{definition}
 For $M = S^3$, \textbf{Integral surgery} along a framed knot $\pmb K$ is Dehn surgery along $\pmb K$ given by a homeomorphism determined by the framing of $\pmb K$.  In this case we will denoted  by $M_{\pmb K}$ the manifold obtained by integral  surgery on the framed knot $\pmb K$ from $S^3$. We may define integral surgery along a framed link by applying the process to each component. The manifold obtained by integral surgery along a framed link from $S^3$ is denoted by $M_{ \pmb L }$.
 \end{definition}

A relationship between all closed orientable $3$-manifolds and framed links in $S^3$ was first proved by A. H. Wallace in \cite{Wal}, then W. B. R. Lickorish in \cite{Lic1} gave a separate elementary and geometric proof that uses Dehn twist homeomorphisms.

\begin{theorem}[Lickorish-Wallace Theorem] \label{LWT} \cite{Lic1,Wal}
 Every closed orientable $3$-manifold can be obtained from $S^3$ by performing integral surgery on a framed link.
\end{theorem}

By using R. Thom's result in \cite{Thom} we observe Dehn surgery's strong connections to $4$-manifolds.
\begin{theorem}\label{4man}\cite{Thom}
Every closed, connected, $3$-manifold is the boundary of some oriented connected $4$-manifold.
\end{theorem}

We turn our attention to \textbf{even surgery} which is integral surgery in $S^3$ along a framed link whose components all have an even framing number. In R. Kirby's paper on Kirby calculus \cite{Kir}, Kirby found a relationship between even surgery and parallelizable $4$-manifolds. Recall that a $4$-manifold is parallelizable if its tangent bundle is trivial. 

\begin{theorem} \label{evenlinkthrm} \cite{Kir}
Let $\pmb L$ be an even link, then $W_{\pmb L}$ is parallelizable, where $W_{\pmb L}$ is a $4$-manifold determined by a framed link $\pmb L$, by adding $2$-handles to a $4$-ball along $\pmb L$.
\end{theorem}

This relationship becomes more powerful after applying J. Milnor's work in \cite{Mil}; in particular the following theorem. (A simplified proof of Theorem \ref{evensurgerythrm} than that of \cite{Mil} can be found in \cite{Kap}, \cite{Sav}, and \cite{FoMa}).

\begin{theorem}\cite{Mil}\label{evensurgerythrm}
Every closed orientable $3$-manifold can be obtained from $S^3$ by performing integral surgery along an even link.
\end{theorem}

By combining Theorem \ref{evensurgerythrm} and Theorem \ref{evenlinkthrm} we have that every closed orientable $3$-manifold is the boundary of a parallelizable $4$-manifold. This result is the foundation of Fomenko and Matveev's proof of Stiefel's theorem. However, Fomenko and Matveev refer to Husemoller's book \cite{Hus} for the completion of the proof. This leads us to the next section about tangent bundles, stably parallelizable, and parallelizable manifolds for the completion of Fomenko and Matveev's proof and a short proof of Stiefel's theorem.

\section{Stably parallelizable and parallelizable manifolds}\label{section:parallelizable}

Stably parallelizable manifolds, also known as $\pi$-manifolds, were extensively studied in the 1960's and a summary of the findings were presented by E. Thomas in \cite{Th2}. In this section we will use a few key theorems on the tangent bundle of manifolds with boundary and stably parallelizable manifolds that serves vital to the short proof of Stiefel's theorem.

\begin{definition}
A section of the tangent bundle $TM \xrightarrow{\ \pi \ } M$ is called a \textbf{vector field} on $M$. More precisely, a vector field is a smooth function $\nu: M \to TM$ which assigns to each point $p \in M$ a vector $\nu(p)$ tangent to $M$ at $p$ such that $\pi \circ \nu = \text{id}_{M}$.
\end{definition}

In 1966, E. Thomas coined the term ``span" of a manifold in \cite{Th1} while working on stably equivalent vector bundles. 
\begin{definition}
The vector fields $\nu_1, \cdots, \nu_k$ on an $n$-manifold $M$ are called \textbf{linearly independent} if for every $p\in M$ the vectors $\{ \nu_1(p), \nu_2(p), \cdots, \nu_k(p)\}$ are linearly independent.
The maximal number of linearly independent vector fields on $M$ is called the \textbf{span} of $M$ and is denoted by $span(M)$. Clearly, $span(M) \leq n$.
\end{definition}

\begin{definition}
$M$ is \textbf{parallelizable} if the tangent bundle of $M$ is trivial. Equivalently, an $n$-manifold $M$ is parallelizable if and only if $span (M) = n$.
\end{definition}

\begin{definition}
Let $E_1 \xrightarrow{\ \pi_1 \ } M$ and $E_2 \xrightarrow{\ \pi_2 \ } M$ be two vector bundles. We define the \textbf{Whitney sum bundle} (direct sum bundle) to be the vector bundle $E_1 \oplus E_2 \xrightarrow{\ \pi \ } M$  where $E_1 \oplus E_2 = \{ (v_1, v_2) \in E_1 \times E_2 ; \pi_1(v_1)=\pi_2(v_2) \}$ and $\pi(v_1, v_2) = \pi_1(v_1) = \pi_2(v_2)$. For an arbitrary $p \in M$, the fiber at $p$ is $\pi^{-1}( \{p \}) = \pi_1^{-1}( \{ p \}) \oplus \pi_2^{-1}( \{ p \} )$.
\end{definition}

\begin{definition}
Let $M$ be an $n$-manifold, $TM \xrightarrow{\ \pi_1 \ } M$ the tangent bundle over $M$, and $M \times \mathbb{R} \xrightarrow{\ \pi_2 \ } M$ be a trivial line bundle over $M$, then $M$ is called \textbf{stably parallelizable} if the Whitney sum bundle  $TM \oplus (M \times \mathbb{R}) \xrightarrow{\ \pi \ } M$ is a trivial bundle.
\end{definition}

Consider the $n$-sphere smoothly embedded into $\mathbb{R}^{n+1}$ and defined as $S^n := \{ \vec{x} \in \mathbb{R}^{n+1} ; \vec{x} \cdot \vec{x} = 1 \}$. Then the \textbf{normal bundle} of $S^n$ into $\mathbb{R}^{n+1}$, denoted by $NS^n \xrightarrow{\ \pi \ } S^n$, is a vector bundle where the total space is defined by $NS^n = \{ (\vec{x}, t\vec{x}) \in S^n \times \mathbb{R}^{n+1} ;  t \in \mathbb{R}  \}$ with projection map defined by $\pi (\vec{x}, t\vec{x}) = \vec{x}$. The fiber of this bundle is $\mathbb{R}$ since for each point $\vec{p}\in S^n$, we have $\pi^{-1}(\{ \vec{p} \}) = \{(\vec{p}, t \vec{p}); t \in \mathbb{R} \}  \cong \mathbb{R}$. It can be shown that this bundle is bundle isomorphic to the trivial line bundle $S^n \times \mathbb{R}$. Furthermore, the proof that all $n$-spheres are stably parallelizable is usually shown by taking the Whitney sum of the tangent bundle $TS^n$ with the normal bundle $NS^n$ and proving that the resulting bundle is trivial. 

\begin{definition}
Let W be an $n$-manifold and $M$ a manifold embedded into $W$. Consider the tangent bundle of $W$, $TW \xrightarrow{\ \pi \ } W$, then the \textbf{ambient tangent bundle} over $M$, denoted by $TW \vert_M$, is a vector bundle with total space $TW \vert_M = \bigsqcup_{p \in M} T_pW$, the projection map obtained by restricting $\pi$ to $TW \vert_M$, and base space $M$.
\end{definition}

\begin{definition}
Let $W$ be an $n$-manifold and $M$ be a manifold embedded into $W$. The \textbf{normal bundle} of $M$ in $W$ is defined as the following quotient bundle,

$$ NM := \mathlarger{\sfrac{TW \vert_{M}}{ TM}}.$$
Notice that this yields the following bundle isomorphism,
\begin{equation}\label{normalbuniso}
    TW \vert_{M} \cong TM \oplus NM.
\end{equation}
\end{definition}

The following proposition was proved by D. Husemoller in \cite{Hus} by constructing a bundle isomorphism. A different approach to the proof can be done my using strategies in the proof that all $n$-spheres are stably parallelizable. That is, by using the bundle isomorphism in Equation \ref{normalbuniso} and proving that the normal bundle of $M$ is a trivial line bundle.

\begin{proposition}\cite{Hus}\label{boundarystab} Let $W$ be an $(n+1)$-dimensional manifold with boundary $\partial W \neq \varnothing$, and let $M$ be an $n$-manifold such that $M =\partial W$, then
there exists an isomorphism of vector bundles between the ambient tangent bundle over $M$ and the Whitney sum bundle of the tangent bundle of $M$ and a trivial line bundle. That is, 
$$ TW \vert_{M} \cong TM \oplus (M \times \mathbb{R}).$$
\end{proposition}

A nice direct corollary to Proposition \ref{boundarystab} arises when the tangent bundle of a manifold with boundary is trivial. 

\begin{corollary}\label{boundarystable}
Let $W$ be an $n$-dimensional manifold with boundary $\partial W \neq \varnothing$. If the tangent bundle of $W$ is trivial then its boundary is stably parallelizable. 
\end{corollary}

G. E. Bredon and A. Kosinski's theorem in \cite{BK} provided below implies that stably parallelizable implies parallelizable for dimension 3.

\begin{theorem} \cite{BK} \label{snstablyparallel}
Let $M$ be an $n$-dimensional oriented closed stably parallelizable manifold, then $span(M) \geq span(S^n)$.
\end{theorem}

\section{A short proof of Stiefel's theorem}\label{shortproof}
We now present the short proof of Stiefel's theorem by using a special case of even surgeries.

\begin{theorem}\cite{Sti}\label{3manparallel}
Every orientable $3$-manifold is parallelizable. 
\end{theorem}

\begin{proof} Suppose $M$ is a closed orientable $3$-manifolds, then we follow A. T. Fomenko and S. V. Matveev proof in \cite{FoMa} to prove that $M$ is stably parallelizable. By Theorem \ref{evensurgerythrm}, $M$ can be obtained from $S^3$ by integral surgery along an even link. Furthermore, by Theorem \ref{evenlinkthrm} a closed orientable $3$-manifold obtained from $S^3$ by integral surgery along an even link is the boundary of a parallelizable $4$-manifold. Both theorems combined implies that $M$ is the boundary of a parallelizable $4$-manifold. By Corollary \ref{boundarystable}, $M$ is stably parallelizable. Since $S^3$ is parallelizable, then $span(S^3)=3$. By Theorem \ref{snstablyparallel}, $span(M) = 3$, therefore, $M$ is parallelizable. Suppose $M$ is a compact orientable $3$-manifold with non-trivial boundary. Let $X$ be the double of $M$, then $M$ is an embedding of $X$. Since $X$ is a closed orientable $3$-manifold then it is parallelizable. Since the normal bundle of $M$ into $X$, $NM$, is a 0-ranked vector bundle over $M$ (the rank of a vector bundle is the dimension of its fiber.) and $TX \vert_{M} \cong TM \oplus NM$, then the tangent bundle of $M$, $TM$, is trivial and hence $M$ is parallelizable. 
\end{proof}

\section{A remark on $7$-manifolds}\label{7manifolds}

The short proof of Stiefel's theorem raise the following questions: in which dimensions does stably parallelizable imply parallelizable? Is this only true if all orientable manifolds of the specified dimension are already parallelizable? The first question was answered by M. A. Kervaire in \cite{Ker1, Ker2}. One may also find the answer in two famous theorems. The first is on the parallelizability of $n$-spheres by R. Bott and J. Milnor in \cite{BM} and separately by M. A. Kervaire in \cite{Ker}.

\begin{theorem} \cite{BM, Ker}\label{Snparallel}
The sphere $S^n$ is parallelizable if and only if $n = 1, 3,$ or $7$. 
\end{theorem}

The second theorem is on the classification of stably parallelizable manifolds by E. Thomas in \cite{Th1} and separately by G. E. Bredon and A. Kosinski in \cite{BK}. 

\begin{theorem} \cite{Th1, BK} \label{stabparallel}
Let $M$ be a stably parallelizable $n$-manifold, then either $M$ is parallelizable or $span(M) = span(S^n)$.
\end{theorem}

As a result of Theorems \ref{Snparallel} and \ref{stabparallel}, we see that all stably parallelizable $n$-manifolds are parallelizable if and only if $n = 1, 3, 7$. Since all orientable $1$- and $3$-manifolds are parallelizable, then $7$-manifolds are the only manifolds with the property that stably parallelizable implies parallelizable where not all orientable manifolds of this dimension are parallelizable. An example of a $7$-manifold that is not parallelizable can be found by looking at Dold manifolds.

A. Dold in \cite{Dol} introduced a family of manifolds, now known as the Dold manifolds, in his work on odd-dimensional generators for the unoriented cobordism ring.

\begin{definition}
Let $r \geq 0, s \geq 0,$ and $r+s > 0$, then the Dold manifold, denoted by $P(r, s)$, is an $r+2s$-dimensional smooth connected manifold. It is constructed from $S^r \times \mathbb{C}P^s$ by identifying $(x, y) \in S^r \times \mathbb{C}P^s$ with $(-x, \overline{y})$.
\end{definition}

The family of Dold manifolds consist of orientable and non-orientable manifolds. This can be shown by letting $s=0$. Since $P(r, 0) \cong \mathbb{R}P^r$, then $P(r, 0)$ for even $r$ are examples of non-orientable Dold manifolds. H. K. Mukerjee, in his work on classifying the homotopy of Dold manifolds \cite{Muk}, gave a simple condition on $r$ and $s$ to determine the orientability of a Dold manifold.

\begin{theorem}\cite{Muk}\label{OrientabilityDold}
The Dold manifold $P(r, s)$ is orientable if $r+s+1$ is even and is unorientable if $r+s+1$ is odd.
\end{theorem}

J. Korba\v{s}, in his work on the parallelizability and span of the Dold manifolds \cite{Kor}, proved that only six Dold manifolds are stably parallelizable; three of which are parallelizable.  

\begin{theorem}\cite{Kor}\label{stablyparallelDold}
The Dold manifold $P(r, s)$ is stably parallelizable if and only if
$$(r, s) \in \{ (1, 0), (3, 0), (7, 0), (0, 1), (2, 1), (6, 1) \}. $$

Furthermore, only $P(1, 0), P(3, 0),$ and $P(7, 0)$ are parallelizable.
\end{theorem}

There are four $7$-dimensional Dold manifolds, $P(1, 3), P(3, 2), P(5, 1),$ and $P(7, 0)$. By Theorem \ref{OrientabilityDold}, two of the four are orientable; $P(3,2)$ and $P(7, 0)$. Furthermore, by applying Theorem \ref{stablyparallelDold}, $P(3,2)$ is the only orientable $7$-dimensional Dold manifold that is not parallelizable.




\end{document}